\documentclass[a4paper,10pt]{amsart}
\usepackage{amsmath, amsthm, amsfonts, amssymb, mathrsfs, enumerate}

\newcommand{\N}{\mathbb{N}}
\newcommand{\R}{{\mathbb{R}}}
\newcommand{\C}{{\mathbb{C}}}
\newcommand{\Z}{{\mathbb{Z}}}
\newcommand{\dd}{{{\rm d}}}
\newcommand{\ii}{{\rm i}}
\newcommand{\BH}{\mathscr{B}(\mathcal{H})}
\renewcommand{\H}{{\mathcal{H}}}
\newcommand{\Dom}{{\operatorname{Dom}}}

\renewcommand{\Re}{\operatorname{Re}}
\renewcommand{\Im}{\operatorname{Im}}
\newcommand{\cf}{\emph{cf.}}
\newcommand{\ie}{{\emph{i.e.}}}
\newcommand{\eg}{{\emph{e.g.}}}
\newcommand{\sgn}{\operatorname{sgn}}

\begin{document}

\theoremstyle{plain}
\newtheorem{define}{Definition}
\newtheorem{theorem}{Theorem}
\newtheorem{lemma}[theorem]{Lemma}
\newtheorem{criterion}[theorem]{Criterion}
\newtheorem{proposition}[theorem]{Proposition}
\newtheorem{corollary}[theorem]{Corollary}
\renewcommand{\proofname}{Proof}
\newtheorem{example}[theorem]{Example}
\theoremstyle{remark}
\newtheorem{remark}{Remark}

\title[Root system]{Root system of singular perturbations of the harmonic oscillator type operators}

\author{Boris Mityagin}
\address[Boris Mityagin]{
Department of Mathematics,
The Ohio State University,
231 West 18th Ave,
Columbus, OH 43210, USA}
\email{mityagin.1@osu.edu}

\author{Petr Siegl}
\address[Petr Siegl]{Mathematishes Institut, Universit\"at Bern, Alpeneggstrasse 22, 3012 Bern, Switzerland \& On leave from Nuclear Physics Institute ASCR, 25068 \v Re\v z, Czech Republic}
\email{petr.siegl@math.unibe.ch}

\subjclass[2010]{47A55, 47A70, 34L10}

\keywords{non-self-adjoint operators, harmonic oscillator, Riesz basis, quadratic forms, singular potentials}

\date{28th September 2015}

\begin{abstract}
We analyze perturbations of the harmonic oscillator type operators in a Hilbert space $\H$, \ie~of the self-adjoint operator with simple positive eigenvalues $\mu_k$ satisfying $\mu_{k+1}-\mu_k \geq \Delta >0$. Perturbations are considered in the sense of quadratic forms. Under a local subordination assumption, the eigenvalues of the perturbed operator become eventually simple and the root system contains a Riesz basis. 
\end{abstract}

\thanks{
P.S. appreciates the kind hospitality and support of OSU allowing his stays there in November 2012 and July 2013 and acknowledges the SCIEX Program; the work has been conducted within the SCIEX-NMS Fellowship, project 11.263. Till March 2013, P.S. has been supported by a grant within the scope of FCT's project PTDC/\ MAT/\ 101007/2008 and partially by FCT's projects PTDC/\ MAT/\
101007/2008 and PEst-OE/MAT/UI0208/2011.
}

\maketitle

\section{Introduction}

This paper deals with the spectrum and eigensystem of perturbations of a self-adjoint operator $A$ in a Hilbert space $\H$. The operator $A$ is of the one dimensional harmonic oscillator type,
\ie~its eigenvalues are simple, positive and satisfy 
\begin{equation}\label{ass.A}
\begin{aligned}
& A \psi_n = \mu_n \psi_n, \quad \|\psi_n\| =1,
& \mu_1 >0, \quad \exists \Delta> 0, \ \forall n \in \N, \
\mu_{n+1} -\mu_n \geq  \Delta;
\end{aligned}
\end{equation}
see also Remark \ref{rem.sim.as}.
The perturbations are not assumed to be symmetric, therefore
the studied operator $T$ is generically non-self-adjoint (and non-normal), 
hence the spectrum typically does not remain real and the basis property of eigensystem is no longer guaranteed.

The main aim is to extend (to cover in particular the $\delta$ potential) the results of \cite{Adduci-2012-10, Adduci-2012-73,Agranovich-1994-28} on sufficient conditions on perturbations guaranteeing that the eigensystem of the perturbed operator contains a Riesz basis. Problems of this type are studied in many works, both classical ones as \cite{DS3,Kato-1966,Markus-1988} and more recent ones, for instance, \cite{Agranovich-1994-28, Shkalikov-2010-269, Wyss-2010-258, Xu-2005-210, Zwart-2010-249}. 

The essential issue in the analysis is that the gaps between the eigenvalues of the unperturbed operator $A$ do not grow. Assuming that the gaps grow, \ie~$\mu_{n+1}-\mu_n \rightarrow + \infty $ and the perturbation $B$ is bounded, Kato proved, \cf~\cite[Thm. V.4.15a, Lem. V.4.17a]{Kato-1966}, that the system of eigenfunctions of $A+B$, plus possibly finite number of associated functions, contains a Riesz basis. The analogous classical theorem allowing also unbounded perturbations can be found in \cite[Thm. XIX.2.7]{DS3}; nevertheless, the growth condition of the gaps is preserved. Constant gaps are allowed in \cite{Agranovich-1994-28}; however, only bounded perturbations satisfy the assumption in \cite{Agranovich-1994-28} and the result is weaker, since only the Riesz basis with brackets is claimed.

Adduci and Mityagin overcome the problem of the non-growing gaps in the study of the harmonic oscillator, \cf~\cite{Adduci-2012-10}, by 
\begin{enumerate}[a)]
	\item using the Hilbert transform as the important technical tool,
	\item replacing the condition of the boundedness of perturbation $B$ by 
	\begin{equation}\label{trad.ass.B}
	\|B \psi_n \| \rightarrow 0 \ {\rm as} \  n \rightarrow \infty,
	\end{equation}
	\item using the following result of Kato; 
\end{enumerate}
\begin{criterion}[{\cite[Lem.V.4.17a]{Kato-1966}}]\label{lem.Kato}
Let $\{P_j\}_{j=0}^\infty$ be a complete family of orthogonal projections in a Hilbert space $\H$, and let $\{Q_j\}_{j=0}^\infty$ be a family of (not necessarily orthogonal) projections such that $Q_jQ_k = \delta_{jk}Q_j$. Assume that
\begin{equation*}
\begin{aligned}
&{\rm Rank} \, P_0 = {\rm Rank} \, Q_0 < \infty, \\
&\forall f \in \H, \quad \sum_{j=1}^{\infty} \|P_j(Q_j-P_j) f\|^2 \leq c_0 \|f\|^2,
\end{aligned}
\end{equation*}
where $c_0$ is a constant smaller than $1$. Then there is a $W \in \BH$ with $W^{-1} \in \BH$ such that $Q_j=W^{-1} P_j W$ for $j \in \N_0:=\N \cup \{0\}$.
\end{criterion}

The condition \eqref{trad.ass.B} has been called by Shkalikov a \emph{local subordinate condition}, see the discussion in \cite{Shkalikov-2012-18} and also \cite[Sec.1]{Adduci-2012-73} for some explanations on this wording. 
Criterion \ref{lem.Kato} substitutes for the often used Bari--Markus criterion, which is given with more restrictive conditions 
\begin{equation*}
\begin{aligned}
&\sum_{j=0}^\infty \| Q_j -P_j\|^2 < \infty, \\
&{\rm Rank} \, P_j = {\rm Rank} \, Q_j < \infty, \quad j \in \N_0,
\end{aligned}
\end{equation*}
see \eg~\cite[Chap.6, Sec. 5.3, Thm. 5.2]{Gohberg-1969} or \cite{Markus-1988}.

For the harmonic oscillator in $L^2(\R)$, the property of Hermite functions, 
\begin{equation}\label{hn.inf.est}
\max_{x \in \R} |h_n(x)| \leq C \, n^{-\frac{1}{12}}, \quad n \in \N,
\end{equation}
can be used to show that \eqref{trad.ass.B} is satisfied for $B$ being, for instance, a multiplication operator by $V \in L^2(\R)$ what is consistently used in \cite{Adduci-2012-10}.

The results of \cite{Adduci-2012-10} for the harmonic oscillator have been extended in \cite{Adduci-2012-73} to the abstract setting with the possibility of the controlled condensation of eigenvalues, \ie~$\mu_{n+1} -\mu_n \geq \kappa n^{\omega-1}$ with fixed $\kappa >0$ and $\omega >1/2$, or the finite clustering of eigenvalues, \ie~there exist fixed values $q>0$ and $\delta>0$ such that $\mu_{n+q}-\mu_n \geq \delta$ for all $n$. In the latter case, similar results have been obtained in \cite{Shkalikov-2010-269} using different methods. However, the $\delta$ potential is not covered by the assumption \eqref{trad.ass.B} which is essential in \cite{Adduci-2012-10, Adduci-2012-73, Shkalikov-2010-269}.

In this paper, we consider perturbations of $A$ in the sense of quadratic forms. (Such a setting has been considered in \cite{Agranovich-1994-28} under the form $p$-subordination assumption, \cf~\eqref{form.p.sub} below.)
At first we define the quadratic form $t:=a+b$, where $a$ corresponds to $A$ and $b$ is the perturbation. The perturbed operator $T$ is associated with the form $t$, see Section \ref{sec.op.def} for details. Such a framework is one way how to include singular perturbations, \cf~\cite[Chap.VI.3.-4.]{Kato-1966} or \cite[\S.1.]{Simon-1971-21} in self-adjoint setting. Our main example is the harmonic oscillator in $L^2(\R)$ perturbed by the $\delta$ potential with complex coupling.  
We remark that the form $b$ does not need to be closed and therefore it does not need to represent an operator in a considered Hilbert space $\H$, distributional potentials are typical cases. 

A straightforward reformulation of the condition \eqref{trad.ass.B}, coming from \cite{Adduci-2012-10, Adduci-2012-73, Shkalikov-2010-269}, would be 
\begin{equation*}\label{trad.ass.B.ref}
\|B\psi_n\|^2 = \sum_{m=1}^{\infty} |b(\psi_n, \psi_m)|^2 \rightarrow 0 \ {\rm as} \ n \rightarrow \infty.
\end{equation*}
Nevertheless, the analysis of the harmonic oscillator perturbed by the $\delta$ potential, \ie~$b(\phi,\psi)=\phi(0) \overline{\psi(0)} $, reveals that the condition \eqref{trad.ass.B.ref} is not satisfied, \cf~\eqref{hn.expl}--\eqref{hn.bound} in Section \ref{sec.ex}.

Our results are obtained under the assumption
\begin{equation}\label{ass.b}
\exists \alpha >0, \quad \exists M_b >0, \quad \forall m,n \in \N, \quad  |b(\psi_m,\psi_n)| \leq \frac{M_b}{m^{\alpha} n^{\alpha}},  
\end{equation}
see also Remark \ref{rem.4A} at the end of Section \ref{sec.tech.lem}.
This extends the previously considered classes of perturbations. For the harmonic oscillator particularly, it means a step towards singular potentials including the mentioned $\delta$. This paper yields also a partially new version of the proof of the main result in \cite{Adduci-2012-10} for some cases. More precisely, unlike in \cite{Adduci-2012-10, Adduci-2012-73}, where the important technical tool was the Hilbert transform, only the Schur test is used here. The connection to the previous work \cite{Adduci-2012-10} is explained in Section \ref{subsec.V.func}.
Moreover, the result of \cite{Adduci-2012-10} is extended to perturbations by potentials $V \in L(p,\tau)$, $1 \leq p < \infty$, $\tau/4 + t(2p) <0$, \cf~Theorem \ref{thm.V.Lptau} and \eqref{Lptau.def}, \eqref{tp.def} in Section \ref{subsec.V.func}; in particular we obtain the extension to potentials $V \in L^p(\R)$, $1\leq p < 2$.
The overall result on perturbations of the harmonic oscillator is formulated in Corollary \ref{cor.HO} in Section \ref{subsec.sing.V.ho}. Further detailed analysis of the spectrum of a harmonic oscillator perturbed by point interactions can be found in \cite{Mityagin-2014a,Mityagin-2015}.

This paper as well as mentioned previous works aim to find sufficient conditions for the Riesz basisness of the eigensystem. However, the negative results, \ie~the fact that the eigensystem is not a Riesz basis (or even a basis), have been obtained particularly for complex oscillators in \cite{Davies-1999-200, Davies-2000-32, Davies-2004-70, Henry-2012-350, Henry-2014-4, Siegl-2012-86} and just recently in \cite{Mityagin-2013arx}.

The paper is organized as follows. In Section \ref{sec.op.def}, we define the operator $T$ and recall some known facts. The main results on the localization of the spectrum and Riesz basisness of the eigensystem are contained and proven in Section \ref{sec.main.res}. In Section \ref{sec.tech.lem}, we collect technical lemmas used in the proofs of main results. Section \ref{sec.ex} consists of examples and the result for the harmonic oscillator. Conclusions and short discussion are contained in the final Section \ref{sec.concl}.

\section{Definition of the operator and preliminaries}
\label{sec.op.def}

The definition of the operator $T$ is based on the first representation theorem \cite[Thm.VI.2.1]{Kato-1966} that provides the unique correspondence between the $m$-sectorial operator $T$ and the closed sectorial form $t$.
The detailed definition of the operator $T$ can be also found in \cite[Sec.2.]{Agranovich-1994-28}.

The self-adjoint operator $A$ is associated, via the second representation theorem \cite[Thm.VI.2.23]{Kato-1966}, with a quadratic form 
\begin{equation*}
a(\psi,\psi)  := \|A^{\frac 12} \psi\|^2, \ \ 
\Dom(a)  := \Dom(A^{\frac 12}).
\end{equation*}
We consider perturbations by a form $b$ satisfying the condition \eqref{ass.b}.
It follows that $b$ is a form $p$-subordinated perturbation of $a$, \ie~ 
there exist $0 \leq p<1$ and $C>0$ such that  
\begin{equation}\label{form.p.sub}
\forall f \in \Dom(a), \quad |b(f,f)| \leq C \left( a(f,f) \right )^p \|f\|^{2(1-p)},
\end{equation}
see Lemma \ref{lem.sub}. The form $p$-subordination implies the form relative boundedness of $b$ with respect to $a$ with the bound $0$. 
The perturbed operator $T$ is defined as the operator associated, via the first representation theorem \cite[Thm.VI.2.1]{Kato-1966}, with a sectorial form
\begin{equation*}
t:= a + b,  \ \ 
\Dom(t)  = \Dom(a).
\end{equation*}
The domains of $t$ and $a$ are the same due to the $p$-subordination, nevertheless, $\Dom(T)$ and $\Dom(A)$ are typically different. 
The form relative boundedness with the bound 0 together with \cite[Thm.VI.3.4]{Kato-1966} imply that $T$ has a compact resolvent. 

The definition of $T$ can be also reformulated as
\begin{equation*}
T = A^{\frac 12}(I-B(0))A^{\frac 12}.
\end{equation*}
Here $B(z)$, $z \in \C$, is the operator uniquely determined by the bounded form $b((z-A)^{-1/2} \cdot,(\overline{z}-A)^{-1/2} \cdot)$, \ie~$\langle B(z)f,g \rangle = b((z-A)^{-1/2} f,(\overline{z}-A)^{-1/2} g) $ for all $f,g \in \H$. The square root of $z-A$ is defined as 
\begin{equation*}
(z-A)^{-\frac 12} f := \sum_{k \in \N} (z-\mu_k)^{-\frac 12} c_k \psi_k,
\end{equation*}
for $f = \sum_{k \in \N} c_k \psi_k$ and $w^s := |w|^s e^{\ii s \arg w}$, $-\pi < \arg w \leq \pi$, for $0\neq w \in \C$ and $s\in \R$.
For all $z \in \rho(A)$, 
\begin{equation*}
z-T = (z-A)^{\frac 12}(I-B(z))(z-A)^{\frac 12}.
\end{equation*}
This relation yields a suitable representation of the resolvent of $T$, \ie~
\begin{equation}\label{Tz.res.dec}
(z-T)^{-1} = (z-A)^{-\frac 12}(I-B(z))^{-1}(z-A)^{-\frac 12},
\end{equation}
provided $I-B(z)$ is invertible and $z \in \rho(A)$. Formulas of this type are also derived in~\cite[Lem.1.]{Agranovich-1994-28}, \cite[Chap.VI.3.1.]{Kato-1966}.

\begin{remark}\label{rem.sim.as}
We have started with the operator $A$ with eigenvalues satisfying $\mu_{n+1} -\mu_n \geq \Delta$, $\Delta >0$. However, to simplify all formulas, we will assume that $\Delta=1$ and $\mu_1 \geq 1$ in the sequel.
(Without loss of generality this can be always achieved by considering $\Delta^{-1}(A + c I)$ with a suitably chosen $c \in \R_+$.) Eigenvalues $\mu$ then satisfy
\begin{equation}\label{mu.sim}
\mu_k \geq k.
\end{equation}

\end{remark}

\section{Main results}
\label{sec.main.res}

Set 
\begin{equation}\label{Pin.def}
\begin{aligned}
\Pi_0&:=\{ z \in \C: - h < \Re z <  (N+3/2), |\Im z| <  h\} \\
\Pi_k&:=\{ z \in \C: |z-\mu_k| < 1/2 \}, \ \ \Gamma_k:=\partial \Pi_k,\\
\Pi&:=\Pi_0 \cup \Big(\bigcup_{j>N+1}\Pi_j \Big),
\end{aligned}
\end{equation}
where $N \in \N$ and $h>1$ are determined in the following way. The aim is to localize the spectrum of $T$. We will succeed if we guarantee that $\|B(z)\|\leq 1/2$ for $z$ outside of $\Pi$. 

Let $N\equiv N(\alpha)$ be an integer such that
\begin{equation}\label{ass.N}
\forall n>N, \quad M_b \, C(2\alpha) \, \sigma_{2\alpha}(n) \leq  \frac{1}{2}, 
\end{equation}
where $M_b$ is as in \eqref{ass.b} and $C(\alpha),$ $\sigma_{\alpha}(n)$ are introduced in Lemma \ref{lem.sum.est.2} below. $h>1$ is selected such that 
\begin{equation}\label{ass.h}
2 M_b 
\left(
\frac{1}{h}\sum_{k=1}^{N+2} \frac{1}{k^{2\alpha} } 
+
D(2\alpha) \tau_{2\alpha}(h)
\right)
\leq \frac{1}{2},
\end{equation}
where $D(\alpha),$ $\tau_{\alpha}(h)$ are introduced in Lemma \ref{lem.h}.

\begin{proposition}\label{prop.loc}
Let conditions \eqref{ass.A}, \eqref{ass.b} hold and let $N$ and $h$ satisfy the conditions \eqref{ass.N} and \eqref{ass.h}, respectively. Then the eigenvalues of $T$ are contained in the interior of $\Pi$, \cf~\eqref{Pin.def}. Moreover, Riesz projections 
\begin{equation}\label{Pn.SN.def}
\begin{aligned}
S_{N+1}&:=\frac{1}{2\pi \ii} \int_{\Gamma_0}(z-T)^{-1} \dd z, \\
P_j & := \frac{1}{2\pi \ii} \int_{\Gamma_j}(z-T)^{-1} \dd z \ {\rm for} \ \ j>N+1,
\end{aligned}
\end{equation}
are well-defined and 
\begin{equation*}
{\rm Rank} \, S_{N+1} =  N+1, \quad {\rm Rank} \, P_{j} =1 \ {\rm for } \ j>N+1.
\end{equation*}

\end{proposition}
\begin{proof}
At first we show that $(z-T)^{-1}$ is bounded for every $z \notin \Pi_0 \cup_{j>N+1}\Pi_j$. Using the resolvent factorization \eqref{Tz.res.dec}, it suffices to prove that $\|B(z)\| \leq 1/2$. Let $f = \sum_{j=1}^{\infty} f_j \psi_j \in \H$, then

\begin{equation}\label{loc.est.1}
\begin{aligned}
\|B(z)f\|^2  &= 
\sum_{k=1}^{\infty} |\langle B(z) f, \psi_k \rangle|^2 
= 
\sum_{k=1}^{\infty} 
\left| 
\sum_{j=1}^{\infty} \frac{f_j b(\psi_j,\psi_k) }{(z-\mu_j)^{\frac{1}{2}} (z-\mu_k)^{\frac{1}{2}}}
\right|^2 
\\
&\leq 
M_b^2 \sum_{k=1}^{\infty} 
\frac{1}{k^{2\alpha}|\mu_k-z|}
\left(
\sum_{j=1}^{\infty} \frac{|f_j|}{ j^{\alpha}|\mu_j-z|^{\frac{1}{2}}}
\right)^2 
\\
&
\leq 
M_b^2 
\left(
\sum_{k=1}^{\infty} 
\frac{1}{k^{2\alpha}|\mu_k -z|}
\right)^2
\|f\|^2.
\end{aligned}
\end{equation}

Let $\Re z \in [\mu_n-1/2,\mu_{n+1}-1/2]$ and $z \notin \Pi_n$, $n \geq N$. We apply inequalities from Lemma \ref{lem.sum.est.2} and we obtain
\begin{equation}\label{loc.est.2}
\|B(z)\| 
\leq
M_b \, C(2\alpha) \, \sigma_{2\alpha} (n)
\leq \frac{1}{2},
\end{equation}
for $n >N$, where $N$ is chosen above, \cf~\eqref{ass.N}. 

The next step is estimate outside of $\Pi_0$.
If $\Re z = - h$, then
\begin{equation*}
 \|B(z)\|
\leq 
M_b 
\sum_{k=1}^{\infty} \frac{1}{ k^{2\alpha}(k+h) }
\leq 
M_b \, D(2\alpha) \, \tau_{2\alpha}(h)
\leq
\frac{1}{2},
\end{equation*}
for $h$ selected as above, \cf~\eqref{ass.h}.

If $\Re z \in [ - h, (N+3/2)]$, then $|\Im z| \geq h$. We denote $z= s+\ii h$, then
\begin{equation*}
\begin{aligned}
\|B(z)\| 
&\leq 
M_b 
\sum_{k=1}^{\infty} \frac{1}{k^{2\alpha} \sqrt{(\mu_k-s)^2 +h^2 }}
\leq 
2  M_b
\sum_{k=1}^{\infty} \frac{1}{k^{2\alpha} (|\mu_k-s| +h) }
\\
&
\leq 
2  M_b
\left(
\sum_{k=1}^{N+2} \frac{1}{h k^{2\alpha}} +
\sum_{k=N+3}^{\infty} \frac{1}{k^{2\alpha} (k- N-2 +h) }
\right)
\\
&
\leq 
2  M_b
\left(
\frac{1}{h} \sum_{k=1}^{N+2} \frac{1}{k^{2\alpha}} +
D(2\alpha) \tau_{2\alpha}(h)
\right)
\leq
\frac{1}{2}
,
\end{aligned}
\end{equation*}
where we use \eqref{mu.sim}, \eqref{ass.h}, and inequalities 
$$(a+b)/2 \leq \sqrt{a^2 + b^2}, \qquad -h \leq s  \leq N+3/2.$$

The standard argument, based on \cite[Lem.VII.6.7]{DS1}, shows that 
\begin{equation*}
{\rm Trace} \, \frac{1}{2\pi \ii} \int_{\Gamma_n} (z-A)^{-\frac 12}(I-t B(z))^{-1} (z-A)^{-\frac 12} \dd z, \ 0 \leq t \leq 1,   
\end{equation*}
is a continuous integer valued function. Therefore it is constant and the second part of the claim follows.
\end{proof}

The obvious corollary is that the eigenvalues $\lambda_n$ of $T$ become eventually simple (for $n > N+1$) and localized around those of the unperturbed operator $A$, while the first part of the spectrum is localized in $\Pi_0$; it is important that there is only a finite number of eigenvalues in $\Pi_0$. The latter also means that the eigensystem of $T$ contains at most a finite number of root vectors associated with different eigenvalues $\{\lambda_n\}_{n=1}^{N_0}$, $N_0 \leq N+1$, in $\Pi_0$ with algebraic multiplicities $\{m_n\}_{n=1}^{N_0}$, $\sum_{n=1}^{N_0} m_n = N+1$, and the rest, $\{\lambda_n\}_{n>N+1}$, consists of eigenvectors $\{\phi_n\}$ related to eigenvalues in $\cup_{j>N}\Pi_j$ of both algebraic and geometric multiplicity one. 

\begin{remark}
Proposition \ref{prop.loc} serves to define the Riesz projections $S_{N+1}$ and $P_j$ that are further analyzed in Theorem \ref{thm.RB}. If we wish to localize the eigenvalues of $T$ more precisely, we can modify $\Pi_k$, $k>N$, to be circles with radii $r_k \to 0$ instead of $1/2$. Then the straightforward modification of estimates \eqref{loc.est.1}, \eqref{loc.est.2}, and Lemma \ref{lem.sum.est.2} yields that $r_k$ does decay as $\mathcal O (k^{- 2 \alpha})$.

A detailed eigenvalue asymptotics of a harmonic oscillator perturbed by point interactions, see Section 5 below, can be found in \cite{Mityagin-2014a,Mityagin-2015}.
\end{remark}

To formulate the main result we denote $\{P_n^0\}$ the one dimensional spectral projections of $A$ related to eigenvalues $\{\mu_n\}$, \ie~ 
\begin{equation*}\label{sp.proj.A}
P_n^0  := \frac{1}{2\pi \ii} \int_{|z-\mu_n|=\frac \Delta 2}(z-A)^{-1} \dd z, \quad n \in \N.
\end{equation*}

\begin{theorem}\label{thm.RB}
Let conditions \eqref{ass.A} and \eqref{ass.b} hold. 
Then there exists a bounded operator $W$ with bounded inverse such that projectors $\{P_n\}$ and $S_{N+1}$, \cf~\eqref{Pn.SN.def}, satisfy 
$$
W P_n W^{-1} = P_n^0
$$
for all $n > N+1$ and 
$$
W S_{N+1} W^{-1} = \sum_{n=1}^{N+1} P_n^0.
$$
Hence, $\{S_{N+1},P_{N+2},P_{N+3},\dots\}$ is a Riesz system of projectors.
\end{theorem}
\begin{remark}
Projectors $P_n$ are one-dimensional for $n>N+1$ and $S_{N+1}$ has rank $N+1$. Therefore the system of root vectors of $T$ contains a Riesz basis $\{f_n\}_{n=1}^{\infty}$ with $f_n = \phi_n$ for $n>N+1$.
\end{remark}

\begin{proof}
The proof is based on Criterion \ref{lem.Kato}. The spectral projections $P_n^0$ of $A$ form a complete family of orthogonal projections, since $A$ is self-adjoint with discrete spectrum. 

In order to apply Criterion \ref{lem.Kato}, we have to find $N_* > N+2$, $N_* \in \N$, such that
\begin{equation*}
\forall f \in \H, \ \ \sum_{n \geq N_*}^{\infty} \|P_n^0(P_n-P_n^0)f\|^2 \leq \frac{1}{2} \|f\|^2,
\end{equation*}
so we take (in the notation of Criterion \ref{lem.Kato}) $P_0:= \sum_{j=1}^{N_*-1} P_j^0$ and $Q_0= S_{N_*-1}$. The latter has the same rank as $P_0$, \cf~Proposition \ref{prop.loc}.

If $n > N + 2$ and $z\in \Gamma_n$, then we have
\begin{equation}
\label{Pn.est.1}
\begin{aligned}
&\|P_n^0(P_n-P_n^0)f\|^2 \leq \frac{1}{4\pi^2} \left\|\int_{\Gamma_n} ((z-T)^{-1} - (z-A)^{-1} )f \dd z \right\|^2 
\\
& \leq 
\frac{1}{4\pi^2}
\left( \int_{\Gamma_n} \|(z-A)^{-\frac{1}{2}} \| \|(I-B(z))^{-1}\| \|B(z)(z-A)^{-\frac{1}{2}} f\| |\dd z | \right)^2
\\
& \leq \frac{2}{\pi^2} \left( \int_{\Gamma_n} \|B(z)(z-A)^{-\frac{1}{2}} f\| |\dd z| \right)^2,
\end{aligned}
\end{equation}
where we use the factorization of resolvent \eqref{Tz.res.dec} and the bound $\|B(z)\| \leq 1/2$ for $z \in \Gamma_n$, $n>N$, \cf~the proof of Proposition \ref{prop.loc}.
Decomposing $f = \sum_{j=1}^{\infty} f_j \psi_j \in \H$ we obtain
\begin{equation*}
\begin{aligned}
 \|B(z)(z-A)^{-\frac{1}{2}} f\|^2 &= 
\sum_{k=1}^{\infty} 
\left| 
\left\langle 
B(z) (z-A)^{-\frac{1}{2}} 
\sum_{j=1}^{\infty}f_j \psi_j,
\psi_k
\right\rangle
\right|^2
\\
& \leq 
\sum_{k=1}^{\infty} \frac{1}{|\mu_k-z|} 
\left| 
\sum_{j=1}^{\infty} \frac{f_j}{z-\mu_j} b(\psi_j,\psi_k) \right|^2
\\
&
\leq  M_b^2
\sum_{k=1}^{\infty} \frac{1}{k^{2\alpha}|\mu_k-z|} 
\left(
\sum_{j=1}^{\infty} \frac{|f_j|}{j^{\alpha}|\mu_j-z|}  
\right)^2.
\end{aligned}
\end{equation*}
For $n >N +2$, we select $z_n^* \in \Gamma_n$ for which the maximum of the integrand in the last integral in \eqref{Pn.est.1} is attained; notice that $\{z_n^*\}$ depends on $f$. Then we can continue estimates in \eqref{Pn.est.1},
\begin{equation}\label{Pn.est.2}
\begin{aligned}
\|P_n^0(P_n-P_n^0)f\|^2 
\leq 
2 M_b^2  
\sum_{k=1}^{\infty} \frac{1}{k^{2\alpha}|\mu_k-z_n^*|} 
\left( 
\sum_{j=1}^{\infty} \frac{|f_j|}{j^{\alpha}|\mu_j-z_n^*|}  
\right)^2.
\end{aligned}
\end{equation}
We apply Lemma \ref{lem.sum.est.2} on the first sum in \eqref{Pn.est.2},
\begin{equation*}
\begin{aligned}
\|P_n^0(P_n-P_n^0)f\|^2 
\leq 
2 M_b^2 C(2\alpha) \sigma_{2\alpha}(n)
\left( 
\sum_{j=1}^{\infty} \frac{|f_j|}{j^{\alpha}|\mu_j-z_n^*|}  
\right)^2.
\end{aligned}
\end{equation*}

The final step is estimating the sum of $\|P_n^0(P_n-P_n^0)f\|^2$, starting at some $N_1 > N+2$, $N_1 \in \N$. We fix $\omega$, $0< \omega <\min\{\alpha,1/2\}$, and assume that $N_1$ is so large that $\sigma_{\alpha}(n) \leq \sigma_{\alpha}(N_1)$ and $\sigma_{2\alpha}(n)n^{2\omega} \leq \sigma_{2\alpha}(N_1)N_1^{2\omega}$ for $n \geq N_1$.
Then, leaving aside the constant,
\begin{equation*}
\begin{aligned}
&\sum_{n=N_1}^{\infty} 
\sigma_{2\alpha}(n)
\left( 
\sum_{j=1}^{\infty} \frac{|f_j|}{j^{\alpha}|\mu_j-z_n^*|} 
\right)^2
\leq  
\sigma_{2\alpha}(N_1) N_1^{2\omega} 
\sum_{n= N_1}^{\infty} 
\left( 
\sum_{j=1}^{\infty} \frac{|f_j|}{n^{\omega}j^{\alpha}|\mu_j-z_n^*|} 
\right)^2
\\
&
=
\sigma_{2\alpha}(N_1) N_1^{2\omega} 
\|\mathcal{M} \tilde{f}\|^2_{\ell^2(\N)}
,
\end{aligned}
\end{equation*}
where $\mathcal{M}$ is an operator acting in $\ell^2(\N)$ with matrix elements 
\begin{equation*}
\begin{aligned}
\mathcal{M}_{nj}&= 0, & n < N_1, \\
\mathcal{M}_{nj}&=
\frac{1}{n^{\omega}j^{\alpha}|\mu_j-z_n^*|}, & n \geq N_1.
\end{aligned}
\end{equation*}
and $\tilde{f}=\{|f_n|\}_{n\in\N} \in \ell^2(\N)$.

We intend to bound $\|\mathcal{M}\|$ using the Schur test, \cf~\cite{Schur-1911-140} or \cite[Sec. 3]{Dym-2003-210}, \cite[Thm. 5.2]{Halmos-1978-96}. To this end we estimate the following sums by applying Lemma \ref{lem.sum.est.2}
\begin{equation*}
\begin{aligned}
& \sum_{j=1}^{\infty}|\mathcal{M}_{nj}| 
=
\sum_{j=1}^{\infty} \frac{1}{n^{\omega}j^{\alpha}|\mu_j-z_n^*|}
\leq 
\frac{1}{n^{\omega}} C(\alpha) \sigma_{\alpha}(n)
\leq 
C(\alpha) \frac{\sigma_{\alpha}(N_1)}{N_1^{\omega}},
\\
&\sum_{n=1}^{\infty}|\mathcal{M}_{nj}| 
\leq
\frac{1}{j^{\alpha}}
C(\omega)
\sigma_{\omega}(j)
\leq
\frac{C(\omega)}{\omega e}
,
\end{aligned}
\end{equation*}
where we use \eqref{log.est} and that $N_1$ is such that $\sigma_{\alpha}(n) \leq \sigma_{\alpha}(N_1)$. 
The Schur test then yields
\begin{equation*}
\|\mathcal{M}\|^2
\leq 
 \frac{C(\alpha) C(\omega)}{\omega e}   
\frac{\sigma_{\alpha}(N_1)}{N_1^{\omega}}.
\end{equation*}
Therefore, since $\omega < \alpha$, 
\begin{equation*}
\begin{aligned}
\sum_{n\geq N_1}^{\infty} \|P_n^0(P_n-P_n^0)f\|^2 
\leq  
\frac{2 M_b^2 C(2\alpha) C(\alpha) C(\omega)}{\omega e}   
\sigma_{2\alpha}(N_1) \sigma_{\alpha}(N_1) N_1^{\omega} 
\|f\|^2,
\end{aligned}
\end{equation*}
which proves the existence of sought $N_*$ and Theorem \ref{thm.RB}.
\end{proof}

\section{Technical lemmas}
\label{sec.tech.lem}

We collect technical results, mainly estimates on the sums appearing in the proofs of main results. At first we explain in details the form subordination of the perturbation.

\begin{lemma}\label{lem.sub}
Let $A$ and $b$ satisfy \eqref{ass.A} and \eqref{ass.b}.
Then there exist $0 \leq p <1$ and $C>0$ such that  
\begin{equation*}
\forall f \in \Dom(a), \ |b(f,f)| \leq C \left( a(f,f) \right)^p \|f\|^{2(1-p)},
\end{equation*}
\ie~$b$ is $p$-subordinated to $a$. Moreover, for $\alpha \leq 1/2$, $p$ can be selected as $1-2\alpha + \tau$, with $\tau >0$ arbitrarily small. If $\alpha>1/2$, then $b$ is bounded.
\end{lemma}
\begin{proof}
Writing $f= \sum_{j=1}^{\infty} f_j \psi_j$, we get ($0\leq\beta < 1/2$)
\begin{equation*}
\begin{aligned}
|b(f,f)| & = 
\left| 
\sum_{j,k=1}^{\infty} \overline{f}_j f_k b(\psi_j,\psi_k)  \right| 
\leq 
M_b 
\left(
\sum_{j=1}^{\infty} \frac{|f_j|}{j^{\alpha}} 
\right)^2
= 
M_b 
\left(
\sum_{j=1}^{\infty} |f_j|j^{\beta} \frac{1}{j^{\alpha+\beta}}  
\right)^2
\\ 
& 
\leq  
M_b \|A^{\beta} f\|^2 
\sum_{j=1}^{\infty} \frac{1}{j^{2(\alpha+\beta)}}
\leq 
M_b (a(f,f))^{2\beta} \|f\|^{2(1-2\beta)} 
\sum_{j=1}^{\infty} \frac{1}{j^{2(\alpha+\beta)}}, 
\end{aligned}
\end{equation*}
where we used the H\"older inequality in the last step.
For $\alpha \leq 1/2$, we select $\beta=1/2-\alpha + \tau/2$, $\tau >0$ and we receive the claim with $p=1-2\alpha + \tau$. For $\alpha > 1/2$, we can take $\beta=0$ and therefore $b$ is bounded.
\end{proof}

\begin{lemma}\label{lem.sum.est.2}
Let $n \in \N$, $n>1$, $ \Re z_n \in [\mu_n-1/2,\mu_{n+1}-1/2]$, and $z_n \notin \Pi_n$, where $\Pi_n$ is defined in \eqref{Pin.def} and $\gamma >0$. Then
\begin{equation*}
\sum_{k =1}^{\infty} \frac{1}{k^{\gamma}|\mu_k-z_n| } 
\leq
C(\gamma)
\sigma_{\gamma}(n),
\end{equation*}
where $C(\gamma)$ does not depend on $n$ and
\begin{equation*}
\sigma_{\gamma}(n):=
\begin{cases}
n^{-\gamma} \log n, & \gamma \leq 1, 
\\
n^{-1},  & \gamma > 1.
\end{cases}
\end{equation*}

\end{lemma}
\begin{proof}
Using $|\mu_k-z_n|\geq |k-n|/2$ for $k \notin \{n,n+1\}$, $|\mu_n-z_n| \geq 1/2$, and $|\mu_{n+1}-z_n| \geq 1/2$ , we obtain
\begin{equation}
\label{eq.lem.1}
\sum_{k=1}^{\infty} \frac{1}{k^{\gamma}|\mu_k-z_n| } \leq 
2
\left( \sum_{k=1}^{n-1} \frac{1}{k^{\gamma}(n-k) } +
\frac{2}{n^{\gamma}} 
+ \sum_{k=n+2}^{\infty} \frac{1}{k^{\gamma}(k-n) }
\right)
.
\end{equation}
If $f$ is a convex non-negative function in interval $[1,p]$, then
\begin{equation*}
\int_1^p f(x) \dd x \leq \sum_{i=1}^p f(i) \leq f(1) + f(p) + \int_1^p f(x) \dd x.
\end{equation*}
Therefore the first term on the right in \eqref{eq.lem.1} can be estimated as
\begin{equation*}
\begin{aligned}
&\sum_{k=1}^{n-1} \frac{1}{k^{\gamma}(n-k) } 
\leq 
\frac{1}{n-1} + \frac{1}{(n-1)^{\gamma}} + \int_{1}^{n-1} \frac{\dd x}{x^{\gamma} (n-x) }.
\end{aligned}
\end{equation*}
Splitting the integral we obtain
\begin{equation*}
\begin{aligned}
\int_{\frac n2}^{n-1} \frac{\dd x}{x^{\gamma} (n-x) }
&\leq 
\frac{2^{\gamma}}{n^{\gamma}} \int_{\frac n2}^{n-1} \frac{\dd x}{n-x}
\leq
2^{\gamma} \, \frac{\log n}{n^{\gamma}},
\\
\int_1^{\frac n2} \frac{\dd x}{x^{\gamma} (n-x) }
&\leq 
\frac{2}{n} \int_1^{\frac n2} \frac{\dd x}{x^{\gamma}},
\end{aligned}
\end{equation*}
where depending on $\gamma$
\begin{equation*}
\frac{2}{n} \int_1^{\frac n2} \frac{\dd x}{x^{\gamma}} \leq 
\begin{cases}
2^{\gamma}(1-\gamma)^{-1} n^{-\gamma}, & \gamma <1, \\
2 n^{-1} \log n, & \gamma =1, \\
2 (\gamma-1)^{-1} n^{-1}, & \gamma >1.
\end{cases}
\end{equation*}
The third term in \eqref{eq.lem.1} is split as well
\begin{equation*}
\sum_{k=n+2}^{\infty} \frac{1}{k^{\gamma}(k-n) } = 
\sum_{k=n+2}^{2n} \frac{1}{k^{\gamma}(k-n) } + \sum_{k=2n+1}^{\infty} \frac{1}{k^{\gamma}(k-n) }
\end{equation*}
and estimated as
\begin{equation*}
\begin{aligned}
&\sum_{k=n+2}^{2n} \frac{1}{k^{\gamma}(k-n) } 
\leq 
\frac{1}{n^{\gamma}} \sum_{k=n+2}^{2n} \frac{1}{k-n } 
\leq 
\frac{1}{n^{\gamma}} \int_{n+1}^{2n}  \frac{\dd x}{x-n} 
\leq  \frac{\log n}{n^{\gamma}}, 
\\
&\sum_{k=2n+1}^{\infty} \frac{1}{k^{\gamma}(k-n)} 
 \leq
2\sum_{k=2n+1}^{\infty} \frac{1}{k^{\gamma+1}} 
\leq  
2 \int_{2n}^{\infty}  \frac{\dd x}{x^{\gamma +1}}
\leq 
\frac{2^{1-\gamma}}{\gamma}\frac{1}{n^{\gamma}}
,
\end{aligned}
\end{equation*}
where we used $k-n >k/2$ in the second estimate.

Combing all the inequalities and using for $\gamma = 1+\beta >1$ that 
\begin{equation}\label{log.est}
\frac{\log n}{n^{\beta}} \leq \max_{x \geq 0} \frac{x} {e^{\beta x}} = \frac{1}{\beta e}, 
\end{equation}
we obtain the claim.
\end{proof}

\begin{lemma}\label{lem.h}
Let $h>1$. Then
\begin{equation*}
\sum_{k=1}^{\infty} \frac{1}{ k^{\gamma}(k+h) }
\leq 
D(\gamma) \tau_{\gamma}(h),
\end{equation*}
where $D(\gamma)$ does not depend on $h$ and
\begin{equation*}
\tau_{\gamma}(h):=
\begin{cases}
h^{-\gamma}, & {\rm if} \ \gamma < 1, \\
h^{-1} \log h , & {\rm if} \ \gamma = 1, \\
h^{-1}, & {\rm if} \ \gamma > 1.
\end{cases}
\end{equation*}
\end{lemma}
\begin{proof}
Proof is analogous to the one of Lemma \ref{lem.sum.est.2}.
\end{proof}

\begin{remark}\label{rem.4A}
A careful analysis of our proof and proper adjustments in the proofs of 
technical lemmas \ref{lem.sum.est.2} and \ref{lem.h} and in the inequalities related to
\eqref{Pn.est.2} show that we can weaken Condition \eqref{ass.b} in Theorem \ref{thm.RB} by assuming only that 

\begin{equation}\tag{4A}
\exists \beta > \frac 32, \ \
\exists M_b >0, \ \
\forall m,n \in \N, \ \ 
|b(\psi_m,\psi_n)| \leq \frac{M_b}{(\log (m+1) \log (n+1))^\beta}.
\end{equation}
\end{remark}

\section{Examples}
\label{sec.ex}

We start with the analysis of perturbations of the harmonic oscillator in $L^2(\R)$:
\begin{equation}
\label{ho.def}
\begin{aligned}
A &= -\frac{\dd^2}{\dd x^2} + x^2,
&\Dom(A) & = \{ \psi \in W^{2,2}(\R): x^2 \psi \in L^2(\R)\}, \\
a(\psi,\psi) & = \|\psi'\|^2 + \|x \psi\|^2, 
&\Dom(a) & = \{\psi \in W^{1,2}(\R): x \psi \in L^2(\R) \}.
\end{aligned}
\end{equation}
Eigenvalues of $A$ are $\mu_n=2n+1,$ $n \in \N_0$, and eigenfunctions are Hermite functions
\begin{equation}
\label{hn.expl}
h_n(x) = \frac{1}{(2^n n! \sqrt{\pi})^{\frac 12}}e^{-\frac{x^2}{2} }H_n(x), \quad n=0,1,2,\dots.
\end{equation}

\subsection{$\delta$ potential}\label{subsec.delta}
The first example is the perturbation by the $\delta$ potential placed in $x_0$ with coupling $\nu \in \C$, more precisely
\begin{equation*}
b_1(\phi,\psi)  = \nu \, \phi(x_0) \, \overline{\psi(x_0)}, \quad \nu \in \C,
\quad
\Dom(b_1)  = W^{1,2}(\R) .
\end{equation*}
In the following, we estimate $|b_1(h_m,h_n)| = |\nu| |h_m(x_0)| |h_n(x_0)|$.

If $x_0 =0$, the values $h_n(0)$ read, \cf~\cite[Eq.22.4.8, Eq.22.2.14]{Abramowitz-1964-55},
\begin{equation*}
\begin{aligned}
h_{2n-1}(0) &= 0, \ \ 
h_{2n}(0) & = \frac{(-1)^n ((2n)!)^{\frac{1}{2}}}{ \pi^{\frac{1}{4}} 2^n n! }.
\end{aligned}
\end{equation*}
Using the Stirling formula for the factorial,
$$n! = \sqrt{2\pi n} 
\left(
\frac{n}{e}
\right)^n e^{\lambda_n}, \quad \frac{1}{12n+1} < \lambda_n < \frac{1}{12n}, $$
\cf~\cite[Eq.6.1.38]{Abramowitz-1964-55}, \cite[\S 2.9]{Feller-1968}, yields
\begin{equation}\label{hn.bound}
|h_{2n}(0)| =  \frac{1+\mathcal{O}(\frac{1}{n})}{\pi^{\frac{1}{2}} n^{\frac{1}{4}}},
\end{equation}
hence the perturbation $b_1$ satisfies condition \eqref{ass.b} with $\alpha = 1/4$ if $x_0=0$. 

Further analysis, \cf~\cite[p.700]{Askey-1965-87} and further references in 
\cite{Adduci-2012-10}, shows that
\begin{equation}\label{hn.est}
|h_n(x)| \leq C
\begin{cases}
(N^{\frac{1}{2}} (N^{\frac{1}{2}} - x ))^{-\frac{1}{4}}, 
& 0 \leq x \leq N^{\frac{1}{2}} - N^{-\frac{1}{6}},  
\\[1mm]
N^{-\frac{1}{12}}, 
& N^{\frac{1}{2}} - N^{-\frac{1}{6}} \leq x \leq N^{\frac{1}{2}} + N^{-\frac{1}{6}},
\\[1mm] 
\displaystyle
\frac{\exp({- \xi N^{\frac{1}{4} } (x-N^{\frac{1}{2}})^{\frac{3}{2}}})}
{(N^{\frac{1}{2}} (x -N^{\frac{1}{2}}))^{\frac{1}{4}}}
,
& N^{\frac{1}{2}} + N^{-\frac{1}{6}} \leq x \leq (2N)^{\frac{1}{2}},
\\[3mm]
\exp(- \xi x^2), & x \geq (2N)^{\frac{1}{2}},
\end{cases}
\end{equation}
where $N=2n+1$ and $\xi >0$. Therefore we can apply our results, since 
$
|h_n(x_0)| \leq  C n^{-1/4}
$
when $n$ is sufficiently large, \eg~if $N^{1/2} \geq x_0 +1$.

\subsection{Function potential $V$}
\label{subsec.V.func}
We consider a form generated by a function potential
\begin{equation}
\label{V.form}
b_2(\phi,\psi):=\langle V \phi, \psi \rangle.
\end{equation}
At first let $V\in L^p(\R)$, $1 \leq p <\infty$. The condition \eqref{ass.b} holds for $b_2$, since
\begin{equation*}
\begin{aligned}
|\langle V h_m,h_n \rangle| &\leq \| V\|_{L^p} \|h_m h_n \|_{L^q} \leq \|V\|_{L^p} \|h_m h_n\|_{L^\infty}^\frac1p \|h_m h_n\|_{L^1}^\frac1q 
\\
&\leq 
C (mn)^{- \frac{1}{12p}}, \quad m,n \in \N,
\end{aligned}
\end{equation*}
where $1/p +1/q =1$ and we use $\|h_m h_n\|_{L^1} \leq 1$ and \eqref{hn.inf.est}. Hence the statement of Theorem \ref{thm.RB} holds.

A next step is to take $V \in  L(p,\tau)$, where
\begin{equation}\label{Lptau.def}
 L(p,\tau):= \left\{
 v:(1+|x|^2)^{-\frac\tau 2}|v(x)| \in L^p(\R) 
 \right\}, \quad 1 \leq  p < \infty,  \ \tau \geq 0,
\end{equation}
is introduced in \cite{Adduci-2012-10}; notice that $L^p(\R)$ is a special case for $\tau =0$. It has been showed in \cite{Adduci-2012-10} that for $2 \leq p < \infty$ and $n \in \N$, $n>1$,
\begin{equation}\label{Vhn.2.est}
\|V h_n \| \leq C
\begin{cases}
\displaystyle n^{\frac \tau2+t(p)}, & p \neq 4, \\[1mm]
\displaystyle  n^{\frac\tau 2 - \frac 18} \log n, & p=4,
\end{cases}
\end{equation}
with
\begin{equation}\label{tp.def}
t(p) = 
\begin{cases}
-\frac{1}{6}(1-\frac{1}{p}), & 2 \leq p <4, \\[1mm]
-\frac{1}{2p}, & 4 \leq p<\infty,
\end{cases}
\end{equation}
see \cite[Lem.5.2]{Adduci-2012-10} for details. Hence the condition \eqref{trad.ass.B} holds for $V \in L(p,\tau)$ if $2 \leq p < \infty$ and $\tau/2 +t(p) < 0$. (The condition \eqref{trad.ass.B} holds for $V$ also if $\tau/2 +t(p) \leq 0$, $(p,\tau)\neq (4,1/4)$, see  \cite[Prop.5.4]{Adduci-2012-10}.)
Therefore the eigensystem of $T=A+V$ contains a Riesz basis, 
\cf~\cite[Thm.1.3]{Adduci-2012-10}. 
Now we can explain this result of \cite{Adduci-2012-10} as a corollary of Theorem \ref{thm.RB}; indeed
\begin{equation*}
\begin{aligned}
|b_2(h_m,h_n)| &= |\langle V h_m, h_n \rangle| 
\leq 
\min (\|V h_m\|,\|V h_n\|)
\leq
\|V h_m\|^{\frac 12} \|V h_n\|^{\frac 12}
\end{aligned}
\end{equation*}
and the inequalities \eqref{Vhn.2.est} imply \eqref{ass.b} with some $\alpha > 0$.

The approach with forms enables us to include potentials $V \in L(p,\tau)$ with $1 \leq p <2$ and a suitable $\tau$ in our analysis and obtain the following claim.
\begin{theorem}\label{thm.V.Lptau}
Let $A$ be the harmonic oscillator, \cf~\eqref{ho.def}, and $V \in L(p,\tau)$ with $1 \leq p < \infty $, $\tau \geq 0$ and $\tau/4 + t(2p) <0$, where $t(p)$ is defined in \eqref{tp.def}. 
Then the statement of Theorem \ref{thm.RB} holds for $T = A +V$ defined as a sum of forms.
\end{theorem}
\begin{proof}
The proof is based on the inequality 
\begin{equation}\label{hn.q.norm}
\|h_n(x) (1+|x|^2)^\frac{\tilde \tau}{2}\|_{L^{\tilde q}} 
\leq
C 
\begin{cases}
n^{\frac{\tilde \tau}{2} + t(\tilde p)}, & \tilde p \neq 4, \\[1mm]
n^{\frac{\tilde \tau}{2} -\frac 18} \log n, & \tilde p =4,
\end{cases}
\end{equation}
where $1/\tilde q +1/\tilde p = 1/2$ and $n \in \N$, $n>1$, 
proved in \cite[Sec.6.2]{Adduci-thesis}, \cite[Eq.(41)--(46)]{Adduci-2012-10} using \eqref{hn.est}. 
If $V \in L(p,\tau)$, then
\begin{equation}\label{b2.hV.est}
\begin{aligned}
|b_2(h_m,h_n)| & \leq \int_{\R} |V(x)| \, (1+|x|^2)^{-\frac\tau 2} \, |h_m(x)| \, |h_n(x)|\, (1+|x|^2)^{\frac \tau 2} \dd x 
\\
& \leq \|V(x) (1+|x|^2)^{- \frac \tau 2}\|_{L^p} 
\|h_m(x) h_n(x) (1+|x|^2)^{\frac \tau 2}\|_{L^q}
\\
& \leq C \|h_m(x) (1+|x|^2)^{\frac \tau 4}\|_{L^{2q}} \|h_n(x) (1+|x|^2)^{\frac \tau 4}\|_{L^{2q}},
\end{aligned}
\end{equation}
where $1/p +1/q =1$. The inequality \eqref{hn.q.norm} with $\tilde \tau = \tau/2$ and $\tilde q =2q$ yields, for all $n \in \N$, $n>1$,
\begin{equation*}
\|h_n(x) (1+|x|^2)^{\frac \tau 4}\|_{L^{2q}} 
\leq C 
\begin{cases}
n^{\frac \tau 4 + t( 2 p)}, & p \neq 2, \\[1mm]
n^{\frac \tau 4 - \frac 18} \log n, & p = 2.
\end{cases} 
\end{equation*}
Therefore the condition \eqref{ass.b} holds for $b_2$ if $\tau/4 + t(2p) <0$ and Theorem \ref{thm.RB} can be applied.
\end{proof}

\subsection{Singular potential with compact support}
\label{subsec.sing.V.ho}
Another example is a form generated by a singular potential $V \in W^{-s,2}(\R),$ $0 < s < 1/2$, with a compact support, contained in an interval $[-p,p]$,
\begin{equation*}
b_3(\phi,\psi) := (V, \phi \overline{\psi}).
\end{equation*}
For $n$ sufficiently large, \eg~$N^{1/2} \geq 2 p +1$, with $N=2n+1$, we have  
\begin{equation*}
|h_n(x)| \leq  C n^{-\frac{1}{4}}, \quad 
|h_n'(x)| \leq C n^{\frac{1}{4}}, \quad x \in [-3p/2,3p/2].
\end{equation*}
These estimates follow from \eqref{hn.est} and the relation for the derivative of Hermite functions
\begin{equation}\label{hn.prime}
h'_n(x) = \sqrt{\frac{n}{2}} h_{n-1}(x) + \sqrt{\frac{n+1}{2}} h_{n+1}(x).
\end{equation}
Therefore, for $m,n \in \N$, 
\begin{equation*} 
|b_3(h_m,h_n)|  
= 
|(V, h_m h_n)| \leq  C \|h_m h_n \|_{L^2((-p,p))}^{1-s} \|h_m h_n\|_{W^{1,2}((-p,p))}^s 
\leq 
C (mn)^\frac{2s-1}4
\end{equation*}
and the condition \eqref{ass.b} is satisfied for $s < 1/2$. Hence the statement of  Theorem \ref{thm.RB} holds.

The assumption on the compact support of singular potential can be omitted, but the range of $s$ is more restrictive. Namely, if $V \in W^{-s,2}(\R),$ $0 < s < 1/11$, then the statement of Theorem \ref{thm.RB} holds. The proof of this claim is analogous, but we use \eqref{hn.inf.est} and \eqref{hn.prime} or more general inequalities for weighted polynomials, \cf~\cite{Milne-1931-33, Zalik-1983-37}.

\begin{remark}
The condition \eqref{ass.b} puts restrictions on the system of \emph{linear} functionals of $V$. Therefore, if we can represent $V$ as a finite sum of potentials $\{V_j\}$ in such a way that each $V_j$ satisfies condition \eqref{ass.b}, 
then $V$ satisfies the condition \eqref{ass.b} as well and the statement of Theorem \ref{thm.RB} on the Riesz basisness of eigensystem of $T = A +V$ holds. In particular, the following is true.
\end{remark} 

\begin{corollary}\label{cor.HO}
Let $A$ be the harmonic oscillator, \cf~\eqref{ho.def}.
If $V = V_1 + V_2 + V_3$, where $ V_3 \in W^{-s,2}(\R)$, $0 < s < 1/2$, with compact support, $V_2 \in L(p, \tau)$ with $1 \leq p < \infty$, $\tau \geq 0$ and $\tau/4 +t(2p)<0$, \cf~\eqref{Lptau.def}, \eqref{tp.def}, and $V_1 \in L^1(\R)$, then the statement of Theorem \ref{thm.RB} holds for $T = A +V$ defined as a sum of forms.
\end{corollary}

\subsection{Infinite number of $\delta$-potentials}

We can go to infinite sums $V = \sum_j V_j $, but then we have to watch
carefully the behavior of constants $M_j$ in \eqref{ass.b}, \ie~
\begin{equation*}
 | \langle V_j h_m, h_n \rangle | \leq \frac{M_j} {m^\alpha n^\alpha}.
\end{equation*} 
%
To illustrate this point we consider the infinite number of $\delta$ potentials.
\begin{example}
Let $A$ be the harmonic oscillator, \cf~\eqref{ho.def}, and let 
\begin{equation*}
b_4(\phi,\psi) = \sum_{k =-\infty}^{\infty} \nu_k \, \phi(p_k) \, \overline{\psi(p_k)},
\end{equation*}
where points $\{p_k\}_{k \in \Z}$ and coupling constants $\{\nu_k\}_{k\in \Z}$ satisfy $p_k =(\sgn k) \, |k|^{\gamma}$, $0< \gamma \leq 1$, and $\nu_k = \mathcal O(|k|^{-\beta})$, $\beta \geq 0$, respectively. If $\beta + \gamma >1$, then $b_4$ satisfies the condition \eqref{ass.b} and therefore the statement of Theorem \ref{thm.RB} holds.
\end{example}
\begin{proof} 
We intend to determine the relation between $\beta$ and $\gamma$ guaranteeing that $b_4$ satisfies the assumption \eqref{ass.b}. In the first step, we use H\"older inequality and the fact that Hermite functions are either even or odd and obtain
\begin{equation*}
|b_4(h_m,h_n)| \leq 
2 \left(
\sum_{k =0}^{\infty} |\nu_k| |h_m(p_k)|^2
\right)^{\frac 12}
\left(
\sum_{k =0}^{\infty} |\nu_k| |h_n(p_k)|^2
\right)^{\frac 12}.
\end{equation*}
The estimates are based on the behavior of Hermite functions \eqref{hn.est} and are divided into four parts.
For $ k \leq K_1$ with $K_1 \in \N$ such that $(N^{1/2} - N^{-1/6})^{1/\gamma}-1 \leq K_1 \leq (N^{1/2} - N^{-1/6})^{1/\gamma}$, we have
\begin{equation*}
\sum_{k=1}^{K_1} \frac{N^{-\frac 14}}{k^{\beta}(N^{\frac 12}-k^{\gamma})^{\frac 12}  } 
\leq 
 \int_1^{K_1} \frac{N^{-\frac 14} \dd x}{x^\beta (N^{\frac 12}-x^\gamma)^{\frac 12} }
 +  \frac{2N^{-\frac 14}}{p_*^\beta (N^{\frac 12}-p_*^{\gamma})^{\frac 12}  },
\end{equation*} 
where $1 \leq p_* \leq K_1$ is such that the maximum of $x^{-\beta}(N^{1/2}-x^{\gamma})^{-1/2}$ is attained at $x=p_*$.  By a very crude estimate without searching for the actual value of $p_*$, the second term is $O(N^{-1/6})$. The substitution $x^{\gamma} = N^{1/2} y$ in the integral leads to
\begin{equation*}
\int_1^{K_1} \frac{ N^{-\frac 14}\dd x}{x^\beta (N^{\frac 12}-x^\gamma)^{\frac 12} }
\leq 
\frac{N^{ \frac{1-\beta -\gamma}{2\gamma} } }{\gamma} 
\int_{N^{-\frac 12}}^{  1} 
\frac{ y^{ \frac{1-\beta}{\gamma} -1 } \dd y}{ (1-y)^{\frac 12}}.
\end{equation*}
As $N \to + \infty$ the integral is $\mathcal{O}(1)$ for $0< \beta < 1$, $\mathcal{O}(\log N)$ for $\beta =1$, and $\mathcal{O}(N^{\frac{\beta-1}{2\gamma}})$ for $\beta >1$.
Taking into account the behavior of the prefactor,  
assumption \eqref{ass.b} is satisfied if $\beta > 1-\gamma$.

In the second part, we have 
\begin{equation*}
\begin{aligned}
\sum_{k=K_1+1}^{K_2} \frac{N^{-\frac 16}}{k^{\beta} }
&\leq \frac{N^{-\frac 16}}{(K_1+1)^{\beta}} + \int_{K_1+1}^{K_2} \frac{N^{-\frac 16} \dd x}{x^\beta},
\end{aligned}
\end{equation*} 
where $(N^{1/2} + N^{-1/6})^{1/\gamma}-1 \leq K_2 \leq (N^{1/2} + N^{-1/6})^{1/\gamma}$. The first term is $\mathcal{O}(N^{\frac{-\gamma-3\beta}{6 \gamma}})$, giving no condition on $\beta$. For $\beta \neq 1$, the integral can be estimated as
\begin{equation*}
\begin{aligned}
\int_{K_1+1}^{K_2} \frac{N^{-\frac 16} \dd x}{x^\beta}
& \leq
\int_{(N^{\frac 12}-N^{-\frac 16})^{\frac 1\gamma} }^{(N^{\frac 12}+N^{-\frac 16})^{\frac 1\gamma} } \frac{N^{-\frac 16} \dd x}{x^\beta} 
\\
&\leq \frac{N^{-\frac 16}}{|1-\beta|}
\left|
(N^{\frac 12}+N^{-\frac 16})^{\frac{1-\beta}{\gamma}} - (N^{\frac 12}-N^{-\frac 16})^{\frac{1-\beta}{\gamma}}
\right|
\\
& = \frac{N^{\frac{3-3\beta-\gamma}{6 \gamma}}}{|1-\beta|} 
\left|
\frac{2(1-\beta)}{\gamma}N^{-\frac 23} + \mathcal{O}(N^{-\frac 43})
\right|  
= \mathcal{O}(N^{\frac{3-3\beta-5\gamma}{6 \gamma}})
\end{aligned}
\end{equation*}
For $\beta =1$ the integral is $o(N^{-1/6})$. Therefore the condition on $\beta$ reads $\beta > 1 - 5\gamma/3$, which is less restrictive than $\beta >1-\gamma$ from the first part. Using analogous estimates, it can be verified that the conditions on $\beta$ from the third and fourth part are also weaker. Therefore, we will make the assumption $\beta > 1-\gamma$ to guarantee that the condition \eqref{ass.b} holds and therefore the statement of Theorem \ref{thm.RB} holds.
\end{proof}

\subsection{Perturbations of anharmonic oscillators} Now we make a few remarks to extend
the results of \cite{Adduci-2012-73}. We perturb self-adjoint operators in $L^2(\R)$:
\begin{equation}
\label{aho.def}
\begin{aligned}
A &= -\frac{\dd^2}{\dd x^2} + |x|^\beta, \quad \beta \geq 2,
\\
\Dom(A) & = \{ \psi \in W^{2,2}(\R): |x|^\beta \psi \in L^2(\R)\}, \\
a(\psi,\psi) & = \|\psi'\|^2 + \||x|^{\frac \beta 2} \psi\|^2, 
\\
\Dom(a) & = \{\psi \in W^{1,2}(\R): |x|^\frac\beta 2 \psi \in L^2(\R) \}.
\end{aligned}
\end{equation}
The spectrum of $A$ is discrete and eigenvalues are simple. Moreover, eigenvalues $\mu_n$ and normalized eigenfunctions $\phi_n$ satisfy
\begin{align}
&\lim_{n\to \infty} 
\left(
2 \lambda_n^{\frac{2 + \beta}{2\beta}} \Omega_\beta - \left(n+\frac 12 \right)\pi
\right) =0 \quad \textrm{with} \quad \Omega_\beta = 2 \int_0^1(1-x^\beta)^\frac12 \dd x,
\\[3mm]
&|\phi_n(x)| 
 \leq 
C \lambda_n^{\frac{1}{2\beta} - \frac14} 
\begin{cases}
\displaystyle 
(\lambda_n-x^{\beta})^{-\frac 14}, 
& 0 \leq x < \lambda_n^\frac 1\beta - \delta_n, 
\\[2mm]
\displaystyle
(\lambda_n-|\lambda_n^\frac1\beta - \delta_n|^{\beta})^{- \frac14}, 
& \lambda_n^\frac 1\beta - \delta_n \leq x \leq \lambda_n^\frac 1\beta + \delta^{(1)}_n, 
\\[2mm]
\displaystyle
\frac{
\exp\left(-\int_{\lambda_n^{\frac 1\beta}}^x (t^\beta - \lambda_n)^{\frac 12}\dd t\right)}
{(x^\beta-\lambda_n)^{\frac 14}}, 
& x > \lambda_n^\frac1\beta + \delta^{(1)}_n, 
\end{cases}
\end{align}
where $\delta_n$ and $\delta^{(1)}_n$ are defined by
\begin{equation}
\int_{\lambda_n^\frac 1 \beta - \delta_n}^{\lambda_n^\frac 1 \beta} (\lambda_n - x^\beta)^\frac12 \dd x =1,
\quad 
\int_{\lambda_n^\frac 1 \beta}^{\lambda_n^\frac 1 \beta+ \delta^{(1)}_n} (x^\beta - \lambda_n)^\frac12 \dd x =1
\end{equation}
and satisfy $\delta_n, \delta^{(1)}_n = \mathcal O(\lambda_n^{-\frac{1-\beta}{3\beta}})$ 
and $\delta_n^{-1}, (\delta^{(1)}_n)^{-1} = \mathcal O(\lambda_n^\frac{1-\beta}{3\beta})$, 
see \cite{Titchmarsh-1954-5,Giertz-1964-14}, \cite[Sec.22.27]{Titchmarsh-1958-book2} and \cite{Mityagin-2013arx} for more details.

The growth condition on eigenvalues implies the existence of constant $C>0$ and $N \in \N$, both depending on $\beta$, such that for all $n > N$
\begin{equation*}
\lambda_{n+1} - \lambda_n \geq C n^{\alpha-1}, \quad \textrm{with} \quad \alpha = \frac{2\beta}{\beta+2},
\end{equation*}
\cf~\cite[Sec.8.2]{Adduci-2012-73}.
Therefore, if $\beta \geq 2$, the growth condition on $\mu_n$ in \eqref{ass.A} is satisfied.

As a perturbation we consider form $b_2$, \cf~\eqref{V.form}, generated by a function potential $V \in L(p,\tau)$, \cf~\eqref{Lptau.def}, and formulate an analogous result to Theorem \ref{thm.V.Lptau} for harmonic oscillator.
\begin{theorem}\label{thm.aho.Lptau}
Let $A$ be the anharmonic oscillator, \cf~\eqref{aho.def}, and $V \in L(p,\tau)$ with $1 \leq p < \infty$, $\tau \geq 0$ and $\tau / (\beta+2) + t(2p,\beta) <0$, 
where 
\begin{equation}
t(p,\beta) :=
\begin{cases}
- \frac16 \left(1 - \frac{4(\beta-1)}{\beta+2} \frac 1p \right), 
& 1 \leq p <2,
\\ 
-\frac{2}{\beta+2} \frac{1}{p},
& 2 \leq  p \leq \infty,
\end{cases}
\end{equation}
Then the statement of Theorem \ref{thm.RB} holds for $T = A +V$ defined as a sum of forms.
\end{theorem}
\begin{proof}
The proof follows the lines of the one of Theorem \ref{thm.V.Lptau}, but the estimate for $\phi_n$, $n\in \N$, $n>1$, is used,
\begin{equation}\label{phi.Lq}
\begin{aligned}
&\|\phi_n(x)(1+|x|^2)^{\frac{\tilde \tau}2} \|_{L^{\tilde q}} 
\leq
C 
\begin{cases}
\displaystyle
n^{\frac{2}{\beta+2} \left(\tilde \tau - \frac 1{\tilde p} \right)  }, & 2 \leq  \tilde p < 4, 
\\
\displaystyle
n^{\frac{2}{\beta+2} \left(\tilde \tau - \frac 14 \right) } \, \log n, & \tilde p = 4, 
\\
\displaystyle
n^{-\frac 16 + \frac{2}{\beta+2} \left(\tilde \tau + \frac{\beta-1}{3\tilde p} \right)  }, & \tilde p > 4, 
\end{cases}
\end{aligned}
\end{equation}
where $1/ \tilde q + 1 / \tilde p =1/2 $. If $V \in L(p,\tau)$, then the same procedure as in \eqref{b2.hV.est} gives
\begin{equation*}
|b_2(\phi_m,\phi_n)| \leq C \|\phi_m(x) (1+|x|^2)^{\frac \tau4} \|_{L^{2q}} \|\phi_n(x) (1+|x|^2)^{\frac \tau 4} \|_{L^{2q}},
\end{equation*}
where $1/p + 1/q =1$, and the claim is then obtained from inequalities \eqref{phi.Lq} with $\tilde \tau = \tau/2$ and $\tilde q = 2q$, \ie~$\tilde p =2p$. 
\end{proof}

\section{Conclusions}
\label{sec.concl}

The positive results, \ie~the claims that the eigensystem of $T$ contains a Riesz basis, are obtained if \emph{the local subordination} condition in the sense of operators \eqref{trad.ass.B}, used in \cite{Adduci-2012-10, Adduci-2012-73, Shkalikov-2010-269}, or in the sense of forms \eqref{ass.b} is satisfied. 
The usual subordination condition in the sense of operators or forms, \cf~\eqref{form.p.sub}, is 
used in previous works \cite{Markus-1988}, \cite{Agranovich-1994-28}, \cite{Wyss-2010-258} where no  claims on the Riesz basisness for the perturbations of the harmonic oscillator are given (except a weaker result of Riesz basis of subspaces for bounded perturbations, \cf~\cite{Agranovich-1994-28}). As it follows from \cite[Thm.2.6]{Mityagin-2013arx}, eq. (2.18), the example
$$-\frac{\dd^2}{\dd x^2} + x^2 + 2 \ii x, $$
shows that the subordination does not help to claim even the basisness of eigensystem, since in this case the perturbation $2\ii x$ is subordinated, however, the norms of Riesz projections $P_n$ grows in $n$, namely
$$\lim_{n \to \infty} \frac{1}{\sqrt{n}} \log \|P_n\| = 2^{\frac 32}.$$
In fact, various rates of spectral projection growth can be obtained by subordinated perturbations of harmonic and anharmonic oscillators, \cf~\cite{Mityagin-2013arx}.

{\footnotesize
\bibliographystyle{acm}
\bibliography{C:/Data/00Synchronized/references}
}

\end{document}